\documentclass[12pt]{article}

\usepackage{amsmath,amssymb,amsthm,amsfonts,diagrams}
\usepackage{nicematrix,tikz}
\usetikzlibrary{fit}
\tikzset{highlight/.style={rectangle,
fill=gray!50,
rounded corners = 0.5 mm, 
inner sep=1pt,
fit=#1}}

\def\Hom{{\rm Hom}}
\def\ind{{\rm ind}}
\def\ad{\mathop{\rm ad}}
\def\im{\mathop{\rm im}}

\def\phi{\varphi}
\def\g{\mathfrak g}

\def\h{\mathfrak h}
\def\F{\mathbb F}

\makeatletter \let\@@pmod\pmod
\DeclareRobustCommand{\pmod}{\@ifstar\@pmods\@@pmod}
\def\@pmods#1{\mkern4mu({\operator@font mod}\mkern 6mu#1)}
\makeatother

\makeatletter
\def\foo#1\endgraf\unskip#2\foo{\def\row@to@buffer{#1\endgraf\unskip\unskip#2}}
\expandafter\foo\row@to@buffer\foo
\makeatother

\newtheorem{theorem}{Theorem}[section]
\newtheorem{lemma}[theorem]{Lemma}

\newtheorem{definition}[theorem]{Definition}
\newtheorem{rem}[theorem]{Remark}

\theoremstyle{remark}

\providecommand{\keywords}[1]{\noindent{Keywords:} #1}
\providecommand{\classify}[1]{\noindent{Mathematics Subject
    Classification:} #1}

\title{On the Cohomology of Restricted Heisenberg Lie Algebras}

\author{Tyler J. Evans \\
  Department of Mathematics \\
  California State Polytechnic University - Humboldt\\
  Arcata, CA 95521 USA\\
  evans@humboldt.edu \and
  Alice Fialowski \\
  Faculty of Informatics\\
  E\" otv\" os Lor\' and University \\
  Budapest, Hungary\\
  fialowski@inf.elte.hu,
  alice.fialowski@gmail.com \and
Yong Yang\\
College of Mathematics and System Science\\
Xinjiang University\\
Urumqi 830046, China\\
yangyong195888221@163.com}

\date{}

\begin{document}
\maketitle

\begin{abstract}
  We show that the Heisenberg Lie
  algebras over a field $\F$ of
  characteristic $p>0$ admit a family of restricted Lie
  algebras, and we classify all such non-isomorphic restricted Lie algebra structures.
  We use the ordinary 1- and 2-cohomology spaces with
  trivial coefficients to compute the restricted 1- and 2-cohomology
  spaces of these restricted Heisenberg Lie algebras.  We describe the
  restricted 1-dimensional central extensions, including explicit
  formulas for the Lie brackets and $\cdot^{[p]}$-operators. 
\end{abstract}

{\footnotesize

\classify{17B50; 17B56}}

\section{Introduction}
Heisenberg Lie algebras have attracted special attention in modern
mathematics and physics because of their applications in the
commutation relations in quantum mechanics. In what follows, we adopt
the following definition:
\begin{definition}\rm
  For a non-negative integer $m$, the \emph{Heisenberg Lie algebra}
  $\mathfrak{h}_{m}$ is the $(2m+1)$-dimensional vector space over a
  field $\F$ spanned by the elements
  \[\{e_1,\ldots , e_{2m},e_{2m+1}\}\] with the non-vanishing Lie
  brackets
  \[[e_i,e_{m+i}]=e_{2m+1},\] for $1\le i\le m$.
\end{definition}

In one of the first explicit computations of the cohomology of a
family of nilpotent Lie algebras, L.J. Santharoubane computed the
cohomology with trivial coefficients of Heisenberg Lie algebras over
fields of characteristic zero \cite{S} (1983). The author shows that
for $n \le m$,
\begin{equation*}
  \dim H^n(\mathfrak{h}_m)={2m \choose n} - {2m \choose {n-2}}
\end{equation*}
in addition to explicitly describing bases for the spaces of cocycles
and coboundaries.  In \cite{SK} (2005), E. Sk\"oldberg used algebraic
Morse theory to compute the Poincar\'e polynomial of the Lie algebra
$\mathfrak{h}_m$ over fields of characteristic $p=2$. G. Cairns and
S. Jambor in \cite{CJ} (2008) gave the $n$th Betti number of
$\mathfrak{h}_m$ over a field of any prime characteristic for
$n \le m$,
\begin{equation*}
  \dim H^n(\mathfrak{h}_m)={{2m } \choose n} - {{2m} \choose {n-2}} +
  \sum_{i=1}^{\lfloor {\frac{n+1}{2p}} \rfloor} {{2m+1} \choose
    {n-2ip+1}} -
  \sum_{i=1}^{\lfloor  {\frac{n-1}{2p}} \rfloor} {{2m+1} \choose {n-2ip-1}}
\end{equation*}

There are also other results related to invariants of Heisenberg
algebras, like gradings and symmetries in \cite{T}, and also the
dimensions of the adjoint cohomology spaces were given in \cite{M},
\cite{CaTi}.

Much less is known about modular Heisenberg algebras over a field of
prime characteristic.

The notion of restricted Lie algebras over fields of positive
characteristic was introduced by Jacobson in 1937 \cite{J}. Since
then, the study of restricted Lie algebras has proved to be fruitful
for several reasons. First, the finite dimensional Lie algebras which
arise ``in nature'' are restricted, e.g. the derivation algebra
$\text{Der}(A)$ of any algebra $A$, the Lie algebra of an algebraic
group, the primitive elements of an irreducible co-commutative Hopf
algebra, etc. Second, certain technical tools are available in a
restricted Lie algebra which are not available in an arbitrary algebra
$\g$. For example, an element $g \in \g$ has a
Jordan-Chevalley-Seligman decomposition into its semisimple and
nilpotent parts, see Seligman \cite{Se}.

Since $\mathfrak{h}_m$ is 2-step nilpotent, $(\ad g)^p=0$ for all
$g \in \mathfrak{h}_m$ so we can use Jacobson's theorem \cite{J} and
define a restricted Lie algebra $[p]$-map on $\mathfrak{h}_m$ by
choosing $e_i^{[p]}\in \mathfrak{h}_m$ so that
$\ad e_i^{[p]}=(\ad e_i)^p=0$ for all $1\le i\le 2m+1$. That is, by
choosing $e_i^{[p]}$ in the center $\F e_{2m+1}$. For each $i$, we
choose a scalar $\lambda_i\in \F$ and set
$e_i^{[p]}=\lambda_i e_{2m+1}$. We let
$\lambda=(\lambda_1,\dots, \lambda_{2m+1})$, and we denote the
corresponding restricted Lie algebra $\mathfrak{h}_m^\lambda(p)$.

Our goal in this note is to classify the family of non-isomorphic
restricted structures on $\mathfrak{h}_m^\lambda(p)$ parameterized by
the elements $\lambda \in \F^{2m+1}$, compute the ordinary and
restricted cohomology spaces (with trivial coefficients)
$H^q(\mathfrak{h}_m^{\lambda}(p))$ and
$H_*^q(\mathfrak{h}_m^{\lambda}(p))$ for $q=1,2$, and give explicit
bases for these spaces. We also give the bracket structures and
$[p]$-operators for the corresponding restricted one-dimensional
central extensions of the restricted Lie algebras.

\section{Restricted Lie algebras and cohomology}

In this section, we recall the definition of a restricted Lie algebra,
the Chevalley-Eilenberg cochain complex for Lie algebra cohomology,
and the (partial) cochain complex in \cite{EFu} that we will use to
compute the restricted 1- and 2- cohomology spaces. Everywhere in this
section, $\F$ denotes a field of characteristic
$p>0$ and $\g$ denotes a finite dimensional Lie algebra over $\F$ with
an ordered basis $\{e_1,\dots , e_n\}$.

\subsection{Restricted Lie Algebras}
A \emph{restricted Lie algebra} (also called \emph{p-algebra}) $\mathfrak{g}$ over $\F$ is a Lie
algebra $\mathfrak{g}$ over $\F$ together with a mapping
$\cdot^{[p]}:\g\to\g$, written $g\mapsto g^{[p]}$, such that for all
$a\in\F$, and all $g,h\in \mathfrak{g}$
\begin{itemize}
\item[(1)] $(a g)^{[p]}=a^{p}g^{[p]}$;
\item[(2)]
  $(g+h)^{[p]}=g^{[p]}+h^{[p]}+ \sum\limits_{i=1}^{p-1} s_i(g,h)$,
  where $is_i(g,h)$ is the coefficient of $t^{i-1}$ in the formal expression
  $(\mathrm{ad}(tg+h))^{p-1}(g)$; and
\item[(3)] $(\mathrm{ad}\ g)^{p}=\mathrm{ad}\ g^{[p]}$.
\end{itemize}

The mapping $\cdot^{[p]}$ is called a \emph{$[p]$-operator} on
$\mathfrak{g}$, and a Lie algebra $\g$ is \emph{restrictable} if it is possible to define a
$[p]$-operator on $\g$. We refer the reader to \cite{J} (Chapter V,
Section 7) and \cite{SF} (Section 2.2) for an introduction to
restricted Lie algebras.  Jacobson shows that a finite
dimensional Lie algebra $\g$ is
restrictable if and only if it admits a basis $\{e_1,\dots, e_n\}$
such that $(\ad e_i)^p$ is an inner derivation for
all $1\le i\le n$ \cite{J}. Choosing $e_i^{[p]}\in\g$ with
$\ad e_i^{[p]}=(\ad e_i)^p$ for all $i$ completely determines a
$[p]$-operator.

\subsection{Chevalley-Eilenberg Lie algebra cohomology}

Our primary interest is in classifying (restricted) one-dimensional
central extensions, so we describe only the
Chevalley-Eilenberg cochain spaces $C^q(\g)=C^q(\g,\F)$
for $q=0,1,2,3$ and differentials $d^q:C^q(\g)\to C^{q+1}(\g)$
for $q=0,1,2$ (for details on the Chevalley-Eilenberg cochain
complex we refer the reader to \cite{ChE} or \cite{F}). Set
$C^0 (\g)=\F$ and $C^q (\g)= (\wedge^q\g)^*$ for $q=1,2,3$. We will use the
following bases, ordered lexicographically, throughout the paper.
\begin{align*}
  C^0 (\g):& \{1\}&\\
  C^1 (\g):&  \{e^k\ |\ 1\le k\le n\}&\\
  C^2 (\g):& \{e^{i,j}\ |\ 1\le i<j\le n\}&\\
  C^3 (\g):& \{e^{u,v,w}\ |\ 1\le u<v<w\le n\}&\\
\end{align*}
Here $e^k$, $e^{i,j}$ and $e^{u,v,w}$ denote the dual vectors of the
basis vectors $e_k\in \g$,
$e_{i,j}=e_i\wedge e_j\in \wedge^2\g $ and
$e_{u,v,w}=e_u\wedge e_v\wedge e_w\in \wedge^3\g$, respectively.
The differentials $d^q:C^q (\g)\to C^{q+1}(\g)$ are defined for $\psi\in C^1 (\g)$,
$\phi\in C^2 (\g)$ and $g,h,f\in\g$ by\small
\begin{align*}
  d^0: C^0 (\g)\to C^1 (\g),  &\  d^0=0&\\
  d^1:C^1 (\g)\to C^2 (\g), &\   d^1(\psi)(g\wedge h)=\psi([g,h])&\\
  d^2:C^2 (\g)\to C^3 (\g), &\  d^2(\phi)(g\wedge h\wedge f)=\phi([g,h]\wedge f)-\phi([g,f]\wedge h)+\phi([h,f]\wedge g). &\\
\end{align*}\normalsize
The maps $d^q$ satisfy $d^{q}d^{q-1}=0$ and 
$H^q(\g)=H^q(\g;\F)=\ker(d^q)/\im(d^{q-1})$.

\subsection{Restricted Lie algebra cohomology}

In this subsection, we recall the definitions and results on the
(partial) restricted cochain complex given in \cite{EFu} only for the case
of trivial coefficients. For a more general treatment of the
cohomology of restricted Lie algebras, we refer the reader to
\cite{Fe}.

Given $\phi\in C^2(\g)$, a map $\omega:\g\to\F$ is {\bf
  $\phi$-compatible} if for all $g,h\in\g$ and all $a\in\F$
\\
  
$\omega(a g)=a^p \omega (g)$ and
\begin{equation}
  \label{starprop}
  \omega(g+h)=\omega(g)+\omega(h) + \sum_{\substack{g_i=\mbox{\rm\scriptsize $g$
        or $h$}\\ g_1=g, g_2=h}}
  \frac{1}{\#(g)}\phi([g_1,g_2,g_3,\dots,g_{p-1}]\wedge g_p)
\end{equation}
where $\#(g)$ is the number of factors $g_i$ equal to $g$.

We define
\[\Hom_{\rm Fr}(\g,\F) = \{f:\g\to\F\ |\ f(a x+b
  y)=a^pf(x)+b^pf(y)\}\] for all $a,b\in\F$ and all $x,y\in \g$ to be
the space of {\it Frobenius homomorphisms} from $\g$ to $\F$. A map
$\omega:\g\to\F$ is $0$-compatible if and only if
$\omega\in \Hom_{\rm Fr}(\g,\F)$.

Given $\phi\in C^2(\g)$, we can assign the values of $\omega$
arbitrarily on a basis for $\mathfrak{g}$ and use (1) to define a
$\phi$-compatible map $\omega: \mathfrak{g} \to \F $. We can define
$\omega$ this way because the sum (1) is symmetric in $g$ and $h$
(permuting the $g_i$ does not change the number of $g_i$ equal to
$g=g_1$), both $\phi$ and the Lie bracket are bilinear, and the
exterior algebra is associative ensuring that
$\omega(f+(g+h))=\omega((f+g)+h)$, so that $\omega$ is well defined.
The map $\omega$ is unique because its values are completely
determined by $\omega(e_i)$ and $\phi(e_i,e_j)$ (c.f. \cite{EFi2}). In
particular, given $\phi\in C^2(\g)$, we can define $\tilde\phi(e_i)=0$
for all $i$ and use (\ref{starprop}) to determine the unique
$\phi$-compatible map $\widetilde\phi:\g\to\F$. Note that, in general,
$\widetilde\phi\ne 0$ but $\widetilde\phi (0)=0$. Moreover, If
$\phi_1,\phi_2\in C^2(\g)$ and $a\in\F$, then
$\widetilde{(a\phi_1+\phi_2)} = a\widetilde\phi_1 + \widetilde\phi_2$.
\\

If $\zeta\in C^3(\g)$, then a map $\eta:\g\times \g\to\F$ is {\bf
  $\zeta$-compatible} if for all $a\in\F$ and all $g,h,h_1,h_2\in\g$,
$\eta(\cdot,h)$ is linear in the first coordinate,
$\eta(g,a h)=a^p\eta(g,h)$ and
\begin{align*}
  \eta(g,h_1+h_2) &=
                    \eta(g,h_1)+\eta(g,h_2)-\nonumber \\
                  & \sum_{\substack{l_1,\dots,l_p=1 {\rm or} 2\\ l_1=1,
  l_2=2}}\frac{1}{\#\{l_i=1\}}\zeta (g\wedge
  [h_{l_1},\cdots,h_{l_{p-1}}]\wedge h_{l_{p}}).
\end{align*}
The restricted cochain spaces are defined as $C^0_*(\g)=C^0 (\g)$,
$C^1_*(\g)=C^1 (\g)$,
\[C^2_*(\g)=\{(\phi,\omega)\ |\ \phi\in C^2 (\g), \omega:\g\to\F\
  \mbox{\rm is $\phi$-compatible}\}\]
\[C^3_*(\g)=\{(\zeta,\eta)\ |\ \zeta\in C^3 (\g),
  \eta:\g\times\g\to\F\ \mbox{\rm is $\zeta$-compatible}\}.\] 

For $1\le i\le n$, define $\overline e^i:\g\to\F$ by
$\overline e^i \left(\sum_{j=1}^n a_j e_j\right ) = a_i^p.$ The set
$\{\overline e^i\ |\ 1\le i\le n\}$ is a basis for the space of
Frobenius homomorphisms $\Hom_{\rm Fr}(\g,\F)$. Since
\[\dim C^2_*(\g) = \binom{\dim\g+1}{2}=\binom{\dim\g}{2}+\dim\g,\]
we have that
\begin{equation*}
  \{(e^{i,j},\widetilde{e^{i,j}})\ |\ 1\le i<j\le n\} \cup \{(0,\overline e^i)\ |\ 1\le i\le n\}
\end{equation*}
is a basis for $C^2_*(\g)$. We use this basis in all computations that
follow.

Define $d_*^0=d^0$. For $\psi\in C^1_*(\g)$, define the map
$\ind^1(\psi):\g\to\F$ by $\ind^1(\psi)(g)=\psi(g^{[p]})$. The mapping
$\ind^1(\psi)$ is $d^1(\psi)$-compatible for all $\psi\in C^1_*(\g)$,
and the differential $d^1_*:C^1_*(\g)\to C^2_*(\g)$ is defined by
\begin{equation}
  d^1_*(\psi) = (d^1(\psi),\ind^1(\psi)).
\end{equation}
For $(\phi,\omega)\in C^2_*(\g)$, define the map
$\ind^2(\phi,\omega):\g \times\g \to\F$ by the formula
\[\ind^2(\phi,\omega)(g,h)=\phi(g\wedge
  h^{[p]})-\varphi([g,\underbrace{h,\ldots,h}_{p-1}],h).\] The mapping
$\ind^2(\phi,\omega)$ is $d^2(\phi)$-compatible for all
$\phi\in C^2(\g)$, and the differential $d^2_*:C^2_*(\g)\to C^3_*(\g)$
is defined by
\begin{equation}
  d^2_*(\phi,\omega) =
  (d^2(\phi),\ind^2(\phi,\omega)). 
\end{equation}
These maps $d_*^q$ satisfy $d_*^{q}d_*^{q-1}=0$ and we define
\[H_*^q(\g)=H_*^q(\g;\F)=\ker(d_*^q)/\im(d_*^{q-1}). \]

Note that (with trivial coefficients) if $\omega_1$ and $\omega_2$ are
both $\phi$-compatible, then
$\ind^2(\phi, \omega_1)=\ind^2(\phi, \omega_2)$.

\begin{lemma}
  \label{swap}
  If $(\phi,\omega)\in C^2_*(\g)$ and $\phi=d^1(\psi)$ with
  $\psi\in C^1 (\g)$, then $(\phi,\ind^1(\psi))\in C^2_*(\g)$ and
  $\ind^2(\phi,\omega)=\ind^2(\phi,\ind^1(\psi))$.
\end{lemma}

\begin{proof}
  Since $\ind^1(\psi)$ is $d^1(\psi)$-compatible for all
  $\psi\in C^1 (\g)=C^1_*(\g)$, it follows that if $\phi=d^1(\psi)$,
  then $(\phi,\ind^1(\psi))=(d^1(\psi),\ind^1(\psi))\in C^2_*(\g)$ and
  $\ind^2$ depends only on $\phi$ by the last sentence in the previous
  paragraph.
\end{proof}

If $\g$ is a restricted Lie algebra and $M$ is a restricted
$\g$-module ($g^{[p]}x=g^px$ for all $g\in\g$ and $x\in M$), there is
a six-term exact sequence due to Hochschild \cite[p. 575]{H} that
relates the ordinary and restricted 1- and 2-cohomology spaces:
\begin{diagram}[LaTeXeqno]
  \label{sixterm}
  0 &\rTo &H^1_*(\g,M)&\rTo &H^1(\g,M)&\rTo&\Hom_{\rm Fr}(\g,M^\g) & \rTo \\
  & \rTo & H^2_*(\g,M)&\rTo &H^2(\g,M)&\rTo^\Delta&\Hom_{\rm
    Fr}(\g,H^1(\g,M)). &
\end{diagram}
In the case of trivial coefficients, the map
\[\Delta: H^2(\g)\to\Hom_{\rm Fr}(\g,H^1(\g))\] in
(\ref{sixterm}) is given explicitly by
\begin{equation}\label{delta}\Delta_\phi(g)(h) = \phi(h\wedge
  g^{[p]})-\varphi([h,\underbrace{g,\ldots,g}_{p-1}]\wedge g)
\end{equation}
where $g,h\in\g$ \cite{H,Viv}. It follows that if
$(\phi,\omega)\in C^2_*(\g)$, then
\begin{equation}\label{induce2}\Delta_\phi(g)(h) =
  \ind^2(\phi,\omega)(h,g)
\end{equation}
for all $g,h\in\g$.

\section{Restricted structures on the modular Lie algebra
  $\mathfrak{h}_m(p)$}
In this section, we classify all restricted structures on the Heisenberg Lie
algebra over a field of prime characteristic.

Since $\mathfrak{h}_m$ is 2-step nilpotent, $(\ad g)^p=0$ for all
$g \in \mathfrak{h}_m$ so we can use Jacobson's theorem \cite{J} and
define a restricted Lie algebra $[p]$-map on $\mathfrak{h}_m$ by
choosing $e_i^{[p]}\in \mathfrak{h}_m$ so that
$\ad e_i^{[p]}=(\ad e_i)^p=0$ for all $1\le i\le 2m+1$. That is, by
choosing $e_i^{[p]}$ in the center $\F e_{2m+1}$. For each $i$, we
choose a scalar $\lambda_i\in \F$ and set
$e_i^{[p]}=\lambda_i e_{2m+1}$. Let
$\lambda=(\lambda_1,\dots, \lambda_{2m+1})$, and denote the
corresponding restricted Lie algebra by
$\mathfrak{h}_m^\lambda(p)$. Alternatively, we can say that the restricted
structure is determined by a linear form $\lambda:\h_m\to\F$ defined
by $\lambda(e_i)=\lambda_i$. Below we write $\h_m^\lambda$ in place of
$\h_m^\lambda(p)$ when no confusion can arise. 

\subsection{The case $p > 2$}

In this subsection, we assume $p > 2$. Since $(p-1)$-fold brackets are
all zero, for all $a\in\F$ and all
$g,h\in \mathfrak{h}_{m}^{\lambda}$, we have
$$(a g)^{[p]}=a^{p}g^{[p]},\quad (g+h)^{[p]}=g^{[p]}+h^{[p]},$$
and therefore the $[p]$-operator on $\mathfrak{h}_m^{\lambda}$ is
$p$-semilinear (see also \cite{J}, Chapter 2, Lemma 1.2). From this we
get that if $g=\sum^{2m+1}_{i=1} a_i e_i\in \mathfrak{h}_m^{\lambda}$,
then
\begin{equation}\label{res}
  g^{[p]}=\left(\sum^{2m+1}_{i=1}a^{p}_i\lambda_i\right) e_{2m+1}.
\end{equation}

\begin{theorem}\label{pgreater2}
  Suppose that there are two $p$-operators $\cdot^{[p]}$ and $\cdot^{[p]'}$
  on $\mathfrak{h}_{m}$ defined by two linear forms
  $\lambda,\lambda':\h_m\to\F$, respectively.  Then
  $\mathfrak{h}_{m}^{\lambda}$ and $\mathfrak{h}_{m}^{\lambda'}$ are
  isomorphic if and only if there exist an invertible matrix
  $A=(a_{ij})\in\F^{2m\times2m}$ and
  $k=(k_1,\dots, k_{2m})\in\F^{2m}$, such that
  \begin{itemize}
  \item[(1)] $A\left(
      \begin{array}{cc}
        0 & I_m \\
        -I_m & 0 \\
      \end{array}
    \right)A^{t} =\mu\left(
      \begin{array}{cc}
        0 & I_m \\
        -I_m & 0 \\
      \end{array}
    \right)$;

  \item[(2)] $\mu\left(
      \begin{array}{c}
        \lambda_1 \\
        \vdots \\
        \lambda_{2m} \\
      \end{array}
    \right) =A^{[p]}\left(
      \begin{array}{c}
        \lambda'_1 \\
        \vdots \\
        \lambda'_{2m} \\
      \end{array}
    \right) +\lambda_{2m+1}'\left(
      \begin{array}{c}
        k^{p}_1 \\
        \vdots \\
        k^{p}_{2m}  \\
      \end{array}
    \right)$;

  \item[(3)] $\lambda_{2m+1}=\mu^{p-1}\lambda_{2m+1}'$;
  \end{itemize}
  where $A^{[p]}=(a_{ij}^{p})$ and $(\det A)^2=\mu^{2m}$.
\end{theorem}

\begin{proof}
  If $\Psi: \mathfrak{h}_m^{\lambda} \to \mathfrak{h}_m^{\lambda'}$ is
  a restricted Lie algebra isomorphism, then we can write
  \[\left(
      \begin{array}{c}
        \Psi(e_1) \\
        \vdots \\
        \Psi(e_{2m}) \\
      \end{array}
    \right) =A\left(
      \begin{array}{c}
        e_1 \\
        \vdots \\
        e_{2m} \\
      \end{array}
    \right) +\left(
      \begin{array}{c}
        k_1 \\
        \vdots \\
        k_{2m} \\
      \end{array}
    \right)e_{2m+1}\] and $\Psi(e_{2m+1})=k_{2m+1}e_{2m+1}$ where
  $A$ is an invertible $2m\times 2m$ matrix,
  $k=(k_1,\ldots,k_{2m})\in\F^{2m}$ and $0\ne k_{2m+1}\in\F$.  If
  $g,h \in \mathfrak{h}_{m}^{\lambda}$ and we write
  $g=\sum^{2m+1}_{i=1}a_ie_i$, $h=\sum^{2m+1}_{i=1}b_ie_i$,
  then \[\Psi(g)=(a_1,\ldots,a_{2m})A\left(
      \begin{array}{c}
        e_1 \\
        \vdots \\
        e_{2m} \\
      \end{array}
    \right) +\left(\sum^{2m+1}_{i=1}a_ik_i\right)e_{2m+1}\]
  and
  \[\Psi(h)=(b_1,\ldots,b_{2m})A\left(
      \begin{array}{c}
        e_1 \\
        \vdots \\
        e_{2m} \\
      \end{array}
    \right) +\left(\sum^{2m+1}_{i=1}b_ik_i\right)e_{2m+1}.\] Moreover,
  we have
  \[[\Psi(g),\Psi(h)]=(a_1,\ldots,a_{2m}) A\left(
      \begin{array}{cc}
        0 & I_m \\
        -I_m & 0 \\
      \end{array}
    \right)A^{t} \left(
      \begin{array}{c}
        b_1 \\
        \vdots \\
        b_{2m} \\
      \end{array}
    \right) e_{2m+1}\] and
  \[\Psi([g,h])=(a_1,\ldots,a_{2m}) \left(
      \begin{array}{cc}
        0 & I_m \\
        -I_m & 0 \\
      \end{array}
    \right) \left(
      \begin{array}{c}
        b_1 \\
        \vdots \\
        b_{2m} \\
      \end{array}
    \right)k_{2m+1} e_{2m+1}.\] From $\Psi([g,h])=[\Psi(g),\Psi(h)]$,
  it follows that 
  \[A\left(
    \begin{array}{cc}
      0 & I_m \\
      -I_m & 0 \\
    \end{array}
  \right)A^{t} =k_{2m+1}\left(
    \begin{array}{cc}
      0 & I_m \\
      -I_m & 0 \\
    \end{array}
  \right).\]
We set $\mu=k_{2m+1}$ so that
  \[A\left(
    \begin{array}{cc}
      0 & I_m \\
      -I_m & 0 \\
    \end{array}
  \right)A^{t} =\mu \left(
    \begin{array}{cc}
      0 & I_m \\
      -I_m & 0 \\
    \end{array}
  \right),\]
and taking determinants gives $(\det
  A)^2=\mu^{2m}$.
  On the
  other hand, $\Psi$ preserves the restricted $[p]$-structure so that
  \begin{equation*}
    \Psi(e_{2m+1}^{[p]})=\Psi(e_{2m+1})^{[p]'}\quad \mbox{\rm
      or equivalently}\quad 
    \lambda_{2m+1} \mu e_{2m+1}=\mu^{p}\lambda_{2m+1}'e_{2m+1},
  \end{equation*}
  and
  \begin{align*}
    \left(
    \begin{array}{c}
      \Psi(e_1^{[p]}) \\
      \vdots \\
      \Psi(e_{2m}^{[p]}) \\
    \end{array}
    \right)  & = \left(
               \begin{array}{c}
                 \Psi(e_1)^{[p]'} \\
                 \vdots \\
                 \Psi(e_{2m})^{[p]'} \\
               \end{array}
    \right) \quad \mbox{\rm or equivalently}\\
    \mu\left(
    \begin{array}{c}
      \lambda_1 \\
      \vdots \\
      \lambda_{2m} \\
    \end{array}
    \right) e_{2m+1}&= A^{[p]} \left(
                      \begin{array}{c}
                        \lambda'_1 \\
                        \vdots \\
                        \lambda'_{2m} \\
                      \end{array}
    \right) e_{2m+1}+ \left(
    \begin{array}{c}
      k^{p}_1 \\
      \vdots \\
      k^{p}_{2m} \\
    \end{array}
    \right) \lambda_{2m+1}'e_{2m+1} .
  \end{align*}
  This shows $\lambda_{2m+1}=\mu^{p-1}\lambda_{2m+1}'$ and
  \[\mu\left(
      \begin{array}{c}
        \lambda_1 \\
        \vdots \\
        \lambda_{2m} \\
      \end{array}
    \right) =A^{[p]}\left(
      \begin{array}{c}
        \lambda'_1 \\
        \vdots \\
        \lambda'_{2m} \\
      \end{array}
    \right) +\lambda_{2m+1}'\left(
      \begin{array}{c}
        k^{p}_1 \\
        \vdots \\
        k^{p}_{2m}  \\
      \end{array}
    \right).\]

  Conversely, if there exists an invertible $2m\times 2m$ matrix
  $A$ with $(\det A)^{2}=\mu^{2m}$ and
  $k=(k_1,\ldots,k_{2m})\in\F^{2m}$ which satisfy conditions (1), (2)
  and (3) for some $\mu\in \F$, the argument above is reversible, and we
  obtain an isomorphism between the restricted Lie algebras
  $\mathfrak{h}_m^{\lambda}$ and $\mathfrak{h}_m^{\lambda'}$.
\end{proof}

\begin{rem}\rm If $p>2$ and $m=1$, there are exactly
  $3$ non-isomorphic 3-dimensional restricted Heisenberg algebras
  given by the linear forms $\lambda=0, \lambda=e_1^*$ and
  $\lambda=e_{3}^*$ \cite{EM}.
\end{rem}

\subsection{The case $p = 2$}

If $p=2$, then the $[2]$-operator
$\cdot^{[2]}: \h_m^\lambda(2)\to\h_m^\lambda(2)$ is not
$p$-semilinear, but instead satisfies
\begin{equation}
  \label{pmap2}
  (g+h)^{[2]}=g^{[2]}+h^{[2]}+ [g,h]
\end{equation}
for all $g,h\in\h_m^\lambda(2)$. If
$g=\sum^{2m+1}_{i=1} a_i e_i \in \mathfrak{h}_m^{\lambda}(2)$, then
(\ref{pmap2}) gives
 
 \begin{equation}\label{res2}
   g^{[2]}=\left(\sum^{2m+1}_{i=1}a^{2}_i\lambda_i+\sum^{m}_{j=1} a_j
     a_{m+j} \right)e_{2m+1}.
 \end{equation}
 In order to define an isomorphism, we introduce some notations.  For
 $a=(a_1,\dots,a_{2m})\in\F^{2m}$ we write
 $a_{\uppercase\expandafter{\romannumeral1}}=(a_1,\dots,a_m)$ and
 $a_{\uppercase\expandafter{\romannumeral2}}=(a_{m+1},\dots,a_{2m})$.
 We denote by $E_{ij}$ the $2m\times 2m$ matrix with a single nonzero
 entry equal to $1$ in the $(i,j)$-position.

 \begin{theorem}\label{pis2}
   Suppose that there are two $2$-operators $\cdot^{[2]}$ and
   $\cdot^{[2]'}$ on $\mathfrak{h}_{m}(2)$ defined by linear forms
   $\lambda,\lambda':\h_m(2)\to\F$, respectively.  Then
   $\mathfrak{h}_{m}^{\lambda}(2)$ and
   $\mathfrak{h}_{m}^{\lambda'}(2)$ are isomorphic if and only if
   there exists an invertible $2m\times 2m$ matrix $A=(a_{ij})$ and
   $k=(k_1,\ldots,k_{2m})\in\F^{2m}$ such that conditions (1) and (3)
   in Theorem \ref{pgreater2} hold where $\mu^{2m}=(\det A)^{2}$, and the
   equation
   \begin{itemize}
   \item[(2')]
     \begin{align*}
       \mu a\left(\sum_{i=1}^{2m} \lambda_i
       E_{ii}\right)&a^{T}+\mu
                      a_{\uppercase\expandafter{\romannumeral1}}a^{T}_{\uppercase\expandafter{\romannumeral2}}=\\ 
                    &aA\left(\sum_{i=1}^{2m}
                      \lambda'_i
                      E_{ii}\right)(aA)^{T}+(aA)_{\uppercase\expandafter{\romannumeral1}}(aA)^{T}
                      _{\uppercase\expandafter{\romannumeral2}}+\lambda'_{2m+1}\left(ak^{T}\right)^2
     \end{align*}
   \end{itemize}
   holds for any $a\in \F^{2m}$ .
   
 \end{theorem}

\begin{proof}
  If
  $\Psi: \mathfrak{h}_m^{\lambda}(2) \rightarrow
  \mathfrak{h}_m^{\lambda'}(2)$ is a restricted Lie algebra
  isomorphism, then using the same notation, the proof of
  Theorem~\ref{pgreater2} still shows that $\Psi$ satisfies conditions (1)
  and (3). We write
  $g=\sum^{2m+1}_{i=1} a_i e_i \in \mathfrak{h}_m^{\lambda}(2)$, and compute
  \begin{align}\label{2.1}
    \begin{split}
    \Psi(g^{[2]})&=\mu\left(\sum_{i=1}^{2m} a^{2}_i
      \lambda_i+\sum_{j=1}^{m}a_ja_{m+j}\right)e_{2m+1} +
      a_{2m+1}^{2}\lambda_{2m+1}e_{2m+1}\\
    &=\mu\left( a\left(\sum_{i=1}^{2m} \lambda_i
       E_{ii}\right)a^{T}+
     a_{\uppercase\expandafter{\romannumeral1}}a^{T}_{\uppercase\expandafter{\romannumeral2}}\right)e_{2m+1}
   + a_{2m+1}^2\lambda_{2m+1}e_{2m+1}.
   \end{split}
  \end{align}
  On the other hand, if we let $aA=(b_1, b_2, ..., b_{2m})$, then\small
  
  \begin{align}\label{2.2}
    \begin{split}
    \Psi(g)^{[2]'}&=\left(\sum^{2m+1}_{i=1} a_i \Psi(e_i))\right)^{[2]'} \\
                  &=\left(\sum^{2m}_{i=1} a_i \Psi(e_i)\right)^{[2]'}+a_{2m+1}^2 \Psi(e_{2m+1}) ^{[2]'} +0\\
                  &=\left(\sum^{2m}_{j=1} b_j e_j+ \left(\sum^{2m}_{i=1}a_i
                    k_i \right)e_{2m+1}\right)^{[2]'}+a_{2m+1}^2
                  \mu^2\lambda_{2m+1}'e_{2m+1} \\ 
                  &=\left(\sum^{2m}_{j=1} b_j e_j\right)^{[2]'}+
                  \left(\sum^{2m}_{i=1}a_i k_i\right)^2
                  \lambda_{2m+1}'e_{2m+1}+a_{2m+1}^2 \mu^2\lambda_{2m+1}'e_{2m+1} \\ 
                  &=\left(\sum_{j=1}^{2m}
                    b^{2}_j\lambda'_j+\sum_{j=1}^{m} b_j
                    b_{m+j}+\lambda'_{2m+1}\left(\sum_{i=1}^{2m}a_i
                      k_i\right)^2\right) e_{2m+1}\\
                  &+a_{2m+1}^2\mu^2\lambda_{2m+1}' e_{2m+1}\\
                  &=\left(aA\left(\sum_{i=1}^{2m}
                      \lambda'_i
                      E_{ii}\right)(aA)^{T}+(aA)_{\uppercase\expandafter{\romannumeral1}}(aA)^{T}_{\uppercase\expandafter{\romannumeral2}}+\lambda'_{2m+1}\left(ak^{T}\right)^2\right)e_{2m+1}\\
                  &+a_{2m+1}^2\mu^2\lambda_{2m+1}' e_{2m+1}
                \end{split}
  \end{align}
  \normalsize Comparing equation (\ref{2.1}) with (\ref{2.2}), and
  using the equality (3) in Theorem~\ref{pgreater2} we have (2').

 Conversely, if there exists an invertible $2m\times 2m$ matrix
  $A$ with $(\det A)^{2}=\mu^{2m}$ and
  $k=(k_1,\ldots,k_{2m})\in\F^{2m}$ which satisfy conditions (1), (2')
  and (3) for some $\mu\in\F$ and all $a\in\F^{2m}$, the argument above is reversible and we
  obtain an isomorphism between the restricted Lie algebras
  $\mathfrak{h}_m^{\lambda}(2)$ and $\mathfrak{h}_m^{\lambda'}(2)$.
\end{proof}

\begin{rem}\rm 
  For the 3-dimensional Heisenberg algebra in characteristic $p=2$,
  there are two restricted Heisenberg algebras, given by the linear
  forms $\lambda=0$ and $\lambda=e_3^*$ \cite{EM}.
\end{rem}

\section{Restricted cohomology $H^q_*(\h_m^\lambda)$ for $q=1,2$}
The cohomology of Heisenberg Lie algebras $\h_m$ with trivial
coefficients was first studied by Santharoubane in \cite{S} (in
characteristic $0$) and by Cairns and Jambor in \cite{CJ} (in
characteristic $p>0$).

\begin{lemma}\label{H}\cite{S,CJ}
  For $q\leq m$,

  (1) if $\mathrm{char}\ \F=0$,
  \[H^{q}(\h_m)=\bigwedge^{q}(\h_m/\F e_{2m+1})^{\ast}/(d
    e_{2m+1}^{\ast}\wedge\bigwedge^{q-2}(\h_m/\F e_{2m+1})^{\ast});\]

  (2) if $\mathrm{char}\ \F=p>0$,
  \begin{align*}
    H^{q}(\h_m)&=\bigwedge^{q}(\h_m/\F
                 e_{2m+1})^{\ast}/(de_{2m+1}^{\ast}\wedge\bigwedge^{q-2}(\h_m/\F
                 e_{2m+1})^{\ast}) \\
               &\oplus((d e_{2m+1}^{\ast})^{p-1}\wedge
                 e_{2m+1}^{\ast}\wedge\bigwedge^{q-2p+1}(\h_m/\F
                 e_{2m+1})^{\ast})
  \end{align*}
  where $d e_{2m+1}^{\ast}=\sum^{m}_{i=1}e^{i,m+i}$.
  
\end{lemma}

If $p>0$ and $q=1$ or $q=2$, then the second term in part (2) of
Lemma~\ref{H} vanishes so the cohomology spaces do not depend on $p$.
In particular
\[H^1(\h_m)=(\h_m/\F e_{2m+1})^{\ast}\] and
\[H^2(\h_m)=\bigwedge^{2}(\h_m/\F e_{2m+1})^{\ast}/(d
  e_{2m+1}^{\ast}).\]
 
Let us now consider the restricted Lie algebra
$\mathfrak{h}_m^{\lambda}$. By definition, $[\h_m^\lambda,\h_m^\lambda]= \F e_{2m+1}$, and the
$[p]$-operator formulas (\ref{res}) and (\ref{res2}) both imply
$\langle (\h_m^\lambda)^{[p]}\rangle_\F\subseteq
[\h_m^\lambda,\h_m^\lambda]$.  For any restricted Lie algebra $\g$,
\cite[(1.4.2)]{F} and \cite[Theorem 2.1]{H} respectively state that
\[H^1(\g,\F)=(\g/[\g,\g])^* \ \mbox{\rm and}\
  H^1_*(\g,\F)=(\g/([\g,\g]+\langle \g^{[p]}\rangle_\F))^*.\] It
follows that $H^1(\h_m^\lambda,\F)=H^1_*(\h_m^\lambda,\F)$, the
classes of $\{e^1,\dots , e^{2m}\}$ form a basis, the map
$H^1_*(\h_m^\lambda) \to H^1(\h_m^\lambda)$ is an isomorphism, and the
six-term exact sequence (\ref{sixterm}) decouples to the exact
sequence
\begin{diagram}
  0&\rTo&\Hom_{\rm Fr}(\h_m^\lambda,\F) & \rTo & H^2_*(\h_m^\lambda) &
  \rTo &H^2(\h_m^\lambda)&\rTo^\Delta&\Hom_{\rm
    Fr}(\h_m^\lambda,H^1(\h_m^\lambda)).
\end{diagram}
The map $\Hom_{\rm Fr}(\g,\F) \to H^2_*(\g) $ sends $\omega$ to the
class of $ (0,\omega)$.  The image of this map is a $(2m+1)$-dimensional
subspace of $H^2_*(\g)$ spanned by the classes $(0,\overline
e^i)$. This space corresponds to the space of restricted
one-dimensional central extensions that split as ordinary Lie algebra
extensions.

If we let $g=\sum a_i e_i$, $h=\sum b_i e_i$,
$\phi=\sum \sigma_{i,j}e^{i,j}$, and $\omega:\h_m^\lambda\to \F$ be
any $\phi$-compatible map, then (\ref{delta}) together with
(\ref{res}) or (\ref{res2}) gives
\begin{eqnarray*}
  \Delta_\phi(g)(h)  &=& \ind^2(\phi,\omega)(h,g) \\
                         &=&\phi(h\wedge
                             g^{[p]})-\varphi([h,\underbrace{g,\ldots,g}_{p-1}],g)\\
                         &=& \left\{
                             \begin{array}{ll}
                                \left(\sum_{i=1}^{2m+1} a_i^p\lambda_i\right
                               )\left(\sum_{i=1}^{2m} b_i\sigma_{i 2m+1}\right), & \hbox{$p>2$;}\\
                               \left(\sum^{2m+1}_{i=1}a^{2}_i\lambda_i+\sum^{m}_{j=1}
                               a_j a_{m+j}
                               \right)\left(\sum_{i=1}^{2m}
                               b_i\sigma_{i 2m+1}\right) 
                              \\
                               -\varphi([h,g],g), & \hbox{$p=2$.}
                              
                             \end{array}
                                                                                   \right.
\end{eqnarray*}
If $\phi$ is a cocycle, then $\sigma_{i 2m+1}=0$ for all
$1\le i\le 2m$ so $\Delta=0$, and we have an exact sequence
\begin{diagram}[LaTeXeqno]
  \label{ses}
  0&\rTo&\Hom_{\rm Fr}(\h_m^\lambda,\F) & \rTo & H^2_*(\h_m^\lambda
  )&\rTo &H^2(\h_m^\lambda )&\rTo&0.
\end{diagram}

\begin{theorem}
  \label{maintheorem}
  For $m\ge 1$ and any form $\lambda:\h_m\to\F$,
  $\dim H^2_*(\h_m^\lambda)=2m^2+m$ and a basis consists of the
  classes of the cocycles
  \[\{(e^{i,j},\widetilde{e^{i,j}}),
    (e^{m+i,m+j},\widetilde{e^{m+i,m+j}}),
    (e^{i,m+j},\widetilde{e^{i,m+j}}),
    (e^{j,m+i},\widetilde{e^{j,m+i}})\ |\ 1\le i< j\le m\}\]
  \[\bigcup \{(e^{i,m+i},\widetilde{e^{i,m+i}})\ |\ 1\le i\le m-1\}
    \bigcup \{(0,\overline e^i)\ |\ 1\le i\le 2m+1\}.\]
\end{theorem}

\begin{proof}
  The exact sequence (\ref{ses}) gives
  \[H^2_*(\h_m^\lambda)=\Hom_{\rm Fr}(\h_m^\lambda,\F)\oplus
    H^2(\h_m^\lambda).\] We have already remarked that the image of
  the injection $\Hom_{\rm Fr}(\g,\F) \to H^2_*(\g) $ is a
  $(2m+1)$-dimensional subspace of $H^2_*(\g)$ spanned by the classes
  $(0,\overline e^i)$. The splitting map
  $H^2(\h_m^\lambda)\to H^2_*(\h_m^\lambda)$ sends the class of $\phi$
  to the class of $(\phi,\tilde\phi)$. Part (2) of Lemma~\ref{H}
  implies that the classes of
  \[\{e^{i,j}, e^{m+i,m+j}, e^{i,m+j},e^{m+i,j}\ |\ 1\le i< j \le m\}\cup
    \{e^{i,m+i}\ |\ 1\le i\le m-1\}\] form a basis for
  $H^2(\h_m^\lambda)$ so the dimension of $H^2_*(\h_m^\lambda)$ is $4\binom{m}{2}+m-1+2m+1=2m^2+m$. This completes the proof of the theorem.
\end{proof}

\section{Restricted one-dimensional central extensions of
  $\h_m^\lambda$}

One-dimensional central extensions $\mathfrak{G}=\g\oplus\F c$ of an ordinary Lie
algebra $\g$ are parameterized by the cohomology group $H^2(\g)$
[\cite{F}, Chapter 1, Section 4.6], and restricted one-dimensional
central extensions of a restricted Lie algebra $\g$ with $c^{[p]}=0$
are parameterized by the restricted cohomology group $H^2_*(\g)$
[\cite{H}, Theorem 3.3]. If $(\phi,\omega) \in C^2_{*}(\g)$ is a
restricted 2-cocycle, then the corresponding restricted
one-dimensional central extension $\mathfrak{G}=\g\oplus\F c$ has Lie bracket and
$[p]$-operator defined by
\begin{align}\label{genonedimext}
  \begin{split}
  [g,h]_{\mathfrak{G}}&=[g,h]_{\g} + \phi(g\wedge h)c\\
  [g,c]_{\mathfrak{G}}&=0\\
  g^{[p]_{\mathfrak{G}}}&=g^{[p]_{\g}} + \omega(g)c\\
  c^{[p]_{\mathfrak{G}}}&=0
  \end{split}
\end{align}
where $[\cdot,\cdot]_{\g}$ and $\cdot^{{[p]}_{\g}}$ denote the Lie
bracket and $[p]$-operator in $\g$, respectively (\cite{EFu},
equations (26) and (27)). 

With the equations (\ref{genonedimext}) together with
Theorem~\ref{maintheorem} we can explicitly describe the restricted
one-dimensional central extensions of $\h_m^\lambda$. 
Let $g=\sum a_ie_i$ and $h=\sum b_ie_i$ denote two arbitrary elements of
$\h_m^\lambda$.

If $1\le i\le 2m+1$ and $\mathfrak{H}_i=\h_m^\lambda\oplus \F c$ denotes
the one-dimensional restricted central extension of $\h_m^\lambda$
determined by the cohomology class of the restricted cocycle
$(0,\overline e^i)$, then (\ref{genonedimext}) gives the (non-zero) bracket and
$[p]$-operator in $\mathfrak{H}_i$:
\begin{align*}
  \begin{split}
    [g,h]_{\mathfrak{H}_i} & =[g,h]_{\h_m^\lambda};\\
    g^{[p]_{\mathfrak{H}_i}} & = g^{[p]_{\h_m^\lambda}}+ a_i^p c.
  \end{split}
\end{align*}
The central extensions $\mathfrak{H}_i$ form a basis for the $(2m+1)$-dimensional
space of restricted one-dimensional central extensions that split as
ordinary Lie algebra extensions (c.f. \cite{EFi2}).

\paragraph{The case $p>2$.}

If $p>2$, the $(p-1)$-fold bracket in (\ref{starprop}) is a multiple
of $e_{2m+1}$. If $1\le s< t\le 2m$, the map $e^{s,t}$ vanishes on
$e_{2m+1}\wedge \h_m^\lambda$ so that (\ref{starprop}) implies that
$\widetilde{e^{s,t}}$ is $p$-semilinear, and hence
$\widetilde{e^{s,t}}=0$. It follows that for any basis element in
Theorem~\ref{maintheorem} of the form
$(e^{s,t},\widetilde{e^{s,t}})=(e^{s,t},0)$, the corresponding
restricted one-dimensional central extension $\mathfrak{H}_{s,t}$ has
(non-zero) bracket and $[p]$-operator
\begin{align*}
  \begin{split}
    [g,h]_{\mathfrak{H}_{s,t}} & =[g,h]_{\h_m^\lambda}+(a_sb_t-a_tb_s)c;\\
    g^{[p]_{\mathfrak{H}_{s,t}}} & = g^{[p]_{\h_m^\lambda}}.
  \end{split}
\end{align*}
The central extensions $\mathfrak{H}_{s,t}$ form a basis for the
$(2m^2-m-1)$-dimensional space of restricted one-dimensional central
extensions that do not split as ordinary Lie algebra extensions.

\paragraph{The case $p=2$.}
If $p=2$, then (\ref{starprop}) applied to
$(e^{s,t},\widetilde{e^{s,t}})$ reduces to
\begin{align*}
  \widetilde{e^{s,t}}
  (g+h)&=\widetilde{e^{s,t}}(g)+\widetilde{e^{s,t}}
         (h)+e^{s,t}(g\wedge h)\\
       & = \widetilde{e^{s,t}}(g)+\widetilde{e^{s,t}}
         (h) + (a_sb_t-a_tb_s).
\end{align*}
Using this and $\widetilde{e^{s,t}}(e_i)=0$ for all $i$, we have
$\widetilde{e^{s,t}}(g)=a_sa_t$.  Therefore, for any basis element in
Theorem~\ref{maintheorem} of the form
$(e^{s,t},\widetilde{e^{s,t}})$, the corresponding restricted one-dimensional
central extension $\mathfrak{H}_{s,t}$ has (non-zero) bracket and
$[p]$-operator
\begin{align*}
  \begin{split}
    [g,h]_{\mathfrak{H}_{s,t}} & =[g,h]_{\h_m^\lambda}+(a_sb_t-a_tb_s)c;\\
    g^{[p]_{\mathfrak{H}_{s,t}}} & = g^{[p]_{\h_m^\lambda}}+a_sa_tc.
  \end{split}
\end{align*}
Once again, the central extensions $\mathfrak{H}_{s,t}$ form a basis
for the $(2m^2-m-1)$-dimensional space of the restricted
one-dimensional central extensions that do not split as ordinary Lie
algebra extensions.

\paragraph {Acknowledgement.}
The authors would like to express their sincere gratitude to the
anonymous referee for the insightful comments and suggestions, which
have significantly enhanced the clarity of this paper.
\\


\begin{thebibliography}{99}
  


\bibitem{CaTi} L. Cagliero, P. Tirao.  The cohomology of the cotangent
  bundle of Heisenberg groups.  \emph{Adv. Math.}  \textbf{181} (2004)
  276--307.


\bibitem{CJ} G. Cairns, S. Jambor.  The cohomology of the Heisenberg
  Lie algebras over fields of finite characteristic.
  \emph{Proc. Amer. Math. Soc.} \textbf{136} (2008) 3803--3807.

\bibitem{T} A. J. Calder\'{o}n Mart\'{i}n, C. Draper,
  C. Mart\'{i}n-Gonz\'{a}lez, J. M. S\'{a}nchez, Delgado.  Gradings
  and symmetries on Heisenberg type algebras.  \emph{ Linear Algebra
    Appl.} \textbf{458} (2014) 463--502.
    
\bibitem{ChE} C. Chevalley, S. Eilenberg, Cohomology theory of Lie
  groups and Lie algebras, \emph{Trans. Amer. Math. Soc.}  \textbf{63}
  (1948), 85-124.
    
\bibitem{EM} Q. Ehret, A. Maklouf.  Deformations and cohomology of
  restricted Lie-Rinehart algebras in positive characteristic.
  \emph{Arxiv: 2305.16425}, 2023
  

\bibitem{EFi2} T. J. Evans, A. Fialowski.  Restricted one-dimensional
  central extensions of the restricted filiform Lie algebras
  $\mathfrak{m}^{\lambda}_{0}(p)$.  \emph{Linear Algebra Appl.}
  \textbf{565} (2019) 244--257.


\bibitem{EFu}T. J. Evans, D. B. Fuchs.  A complex for the cohomology
  of restricted Lie algebras.  \emph{J. Fixed Point Theory Appl.}
  \textbf{3} (2008) 159--179.

\bibitem{Fe}J. Feldvoss. On the cohomology of restricted Lie
  algebras. \emph{Comm. Algebra} \textbf{19} (1991), 2865--2906.

  
\bibitem{F} D.B. Fuchs. Cohomology of infinite dimensional Lie
  algebras. \emph{Contemporary Soviet Mathematics}, Consultants
  Bureau, New York, 1986.


\bibitem{H} G. Hochschild.  Cohomology of restricted Lie algebras.
  \emph{Amer. J. Math.} \textbf{76} (1954) 555--580.


\bibitem{J} N. Jacobson.  \emph{Lie Algebras.}  John Wiley (1962).


\bibitem{M} L. Magnin.  Cohomologie adjointe des algebres de Lie de
  Heisenberg.  \emph{Comm. Algebra} \textbf{21} (1993) 2101--2129.


\bibitem{S} L. J. Santharoubane.  Cohomology of Heisenberg Lie
  algebras.  \emph{Proc. Amer. Math. Soc.} \textbf{87} (1983) 23--28.

\bibitem{Se} G.B. Seligman.  Modular Lie Algebras, Ergebnisse der
  Mathematikund ihrer Grenzgebiete, Band 40, Springer 1967.


\bibitem{SK} E. Sk\"{o}ldberg.  The homology of Heisenberg Lie
  algebras over fields of characteristic two.
  \emph{Math. Proc. R. Ir. Acad.} \textbf{105A} (2005) 47--49.

\bibitem{SF} H. Strade, R. Farnsteiner.  Modular Lie algebras and
  their representations.  \emph{Monographs and Textbooks in Pure and
    Applied Math. Vol.116}, Marcel Dekker, Inc., New York, 1988


 
\bibitem{Viv} F. Viviani. Restricted infinitesimal deformations of
  restricted simple {L}ie algebras.  \emph{J. of Alg. Appl.}
  \textbf{Vol. 11} No. 5 (2012) 19 pages.




\end{thebibliography}
\end{document}